\documentclass[a4paper, 12pt]{amsart}
\usepackage{amsmath}
\usepackage{amsfonts}
\usepackage{amssymb}
\usepackage[utf8]{inputenc}
\usepackage[T1]{fontenc}
\usepackage{amssymb}
\usepackage{mathrsfs}
\usepackage{amsthm}
\usepackage{xcolor}
\usepackage{float}
\usepackage{tikz}
\usepackage{pgfplots}
\usepgfplotslibrary{fillbetween}

\usepackage{enumerate}

\usepackage{graphicx}

\newtheorem{thrm}{Theorem}[section]
\newtheorem{cor}[thrm]{Corollary}
\newtheorem{lemma}[thrm]{Lemma}
\newtheorem{prop}[thrm]{Proposition}

\newtheorem{claimnumber}{Claim}

\theoremstyle{definition}

\DeclareMathOperator{\cl}{Cl}
\DeclareMathOperator{\Int}{Int}

\DeclareMathOperator{\spa}{span}

\begin{document}

\title[On condensations of $C_p(X)$ spaces onto spaces $(E,w)$]{On condensations of $C_p(X)$ spaces onto Banach spaces with the weak topology}
\author{Miko\l aj Krupski}
\address{Institute of Mathematics\\ University of Warsaw\\ ul. Banacha 2\\
02--097 Warszawa, Poland }
\email{mkrupski@mimuw.edu.pl}
\author{Witold Marciszewski}
\address{Institute of Mathematics\\ University of Warsaw\\ ul. Banacha 2\\
02--097 Warszawa, Poland } \email{wmarcisz@mimuw.edu.pl}
\thanks{The authors were partially supported by the NCN (National Science Centre, Poland) research grant no.\ 2020/37/B/ST1/02613.}

\begin{abstract}
We show that for a broad class of spaces $X$, the space $C_p(X)$ of continuous real-valued functions on $X$ equipped with the pointwise topology maps bijectively onto any separable infinite-dimensional Banach space with the weak topology. We prove that the inverse $\varphi^{-1}$ of any continuous bijection $\varphi: C_p(X)\to (E,w)$, where $X$ is an arbitrary infinite Tychonoff space and $(E,w)$ is an arbitrary normed space with the weak topology, is everywhere discontinuous. This generalizes a recent result of Kania--K\k{a}kol and solves a problem raised by K\k{a}kol, Leiderman and Michalak by providing a complete classification of function spaces $C_p(X)$ that are homeomorphic to Banach spaces with the weak topology.
\end{abstract}



\maketitle

\section{Introduction}

All topological spaces under consideration are assumed to be Tychonoff. For a space $X$ by $C_p(X)$ we denote the space of all real-valued continuous functions on $X$ equipped with the topology of pointwise convergence. The symbol $C^*_p(X)$ stands for the subspace of $C_p(X)$ consisting of all \textit{bounded} continuous functions. If $E$ is a normed space, then by $(E,w)$ we denote $E$ endowed with its weak topology while $(E^\ast, w^\ast)$ denotes its dual equipped with the weak* topology. By a \textit{condensation} we mean a continuous bijection. In this note we prove the following:
\begin{thrm}\label{main}
Let $X$ be an infinite topological space and let $E$ be a normed space. If $\varphi:C_p(X)\to (E,w)$ is a condensation, then its inverse $\varphi^{-1}$ is discontinuous at every point of $E$. An analogous assertion holds if $C_p(X)$ is replaced by $C^*_p(X)$.
\end{thrm}
Theorem \ref{main} generalizes a recent result of Kania and K\k{a}kol \cite{KK}, who proved it in the case when $X$ is compact, thereby answering (in a stronger way) a question posed by the present authors (see \cite[Problem 2]{K}, \cite[Problem 1.2]{KM}) as to whether the space $C_p(K)$ over an infinite compactum $K$ can be homeomorphic to a space $(C(L),w)$, for some compact space $L$. However, a broader problem, due to K\k{a}kol, Leiderman and Michalak -- that of characterizing Banach spaces $E$ for which $(E,w)$ is homeomorphic to $C_p(X)$, for some Tychonoff space $X$ (see \cite[Problem 1.5]{KLM}) -- remained open (cf. \cite[p.~4]{KK}). From Theorem \ref{main}, we immediately deduce a complete classification of homeomorphisms between $C_p(X)$ ($C^*_p(X)$) and $(E,w)$:
\begin{cor}
 Let $X$ be a topological space and let $E$ be a normed space. Then $C_p(X)$ ($C^*_p(X)$) is homeomorphic to $(E,w)$ if and only if $X$ is finite and $\dim E=|X|$.
\end{cor}

Conversely, if we relax homeomorphisms to condensations (i.e. continuous bijections), such mappings exist for a broad class of spaces. We shall prove the following:

\begin{prop}\label{condensation may exist}
Let $X$ be a $\sigma$-compact non-discrete metrizable space and let $E$ be a separable infinite-dimensional Banach space. Then there exists a condensation $\varphi:C_p(X)\to (E,w)$. Similarly, there exists a condensation $\psi:C_p^\ast(X)\to (E,w)$.
\end{prop}

To prove Theorem \ref{main}, we slightly modify the argument of Kania and K\k{a}kol from \cite{KK}. Their approach, in turn, relies heavily on the techniques developed by the present authors in \cite{KM}. In fact, the key observation made in \cite{KK} is that incorporating the classical Klee--Mazur result on linearly independent copies of an interval into the proof of Theorem 4.1 in \cite{KM} yields Theorem \ref{main} in the compact case. However, a careful examination of the proof given in \cite{KK} reveals that similar approach is available in the general case.

\section{Proofs}

First let us recall some auxiliary results that will be used in the proof of Theorem \ref{main}. The following lemma is well known and its proof goes back to the works of Klee and Mazur (see \cite[Ch. VIII, Proposition 2.2]{BP}).

\begin{lemma}(Klee-Mazur)\label{Klee-Mazur}
 Let $E$ be an infinite-dimensional Banach space. Then $(E^\ast,\|\cdot\|)$ contains a linearly independent topological copy $K$ of the unit interval $[0,1]$.
\end{lemma}

The next lemma is folklore and its proof is easy (see \cite[p. 2796]{Sa}).

\begin{lemma}\label{lemma_point-finite}
Let $(A_n)_{n\in \omega}$ be a sequence of pairwise disjoint subsets of a space $X$. Suppose that there is positive integer $M$ such that $|A_n|\leq M$ for all $n$. Then there is a subsequence $(A_{n_k})_{k\in\omega}$ and a sequence $(U_k)_{k\in \omega}$ of open subsets of $X$ such that $A_{n_k}\subseteq U_k$ and the family $\{U_k:k\in \omega\}$ is point-finite in $X$.
\end{lemma}

From Lemma \ref{lemma_point-finite} we immediately obtain the following (cf. \cite[Lemma 3.3]{KK}).

\begin{lemma}\label{lemma2}
 Let $(A_n)_{n\in \omega}$ be a sequence of pairwise disjoint subsets of a space $X$. Suppose that there is positive integer $M$ such that $|A_n|\leq M$ for all $n$. Then there is a subsequence $(A_{n_k})_{k\in\omega}$ and a sequence $(f_k)_{k\in \omega}$ of continuous functions $f_k:X\to [0,1]$ such that $f_k(A_{n_k})\subseteq \{1\}$ and the sequence $(f_k)_{k\in \omega}$ converges pointwise to the constant function equal to $0$.
\end{lemma}

Before proceeding to the proof of Theorem \ref{main} let us fix some notation. By $\underline{0}\in C_p(X)$ we denote the constant function equal to zero. For a finite $F\subseteq X$ and a positive integer $m$ the set
$$O\left(F, m\right)=\{f\in C_p(X):f(F)\subseteq \left(-\tfrac{1}{m},\tfrac{1}{m}\right)\}$$
is a basic neighborhood of $\underline{0}$ in $C_p(X)$. For a Banach space $E$ and $y^\ast\in E^\ast$, we set
$$W(y^\ast)=\{x\in E:|y^\ast(x)|\leq 1\}.$$
For $k\in \omega$, by $[X]^k$ we denote the family of all $k$-element subsets of $X$. If $A\subseteq X$ is a subspace of a space $X$ and $B\subseteq A$, then the symbols $\Int_A(B)$, $\cl_A(B)$ stand for the interior, respectively, the closure of $B$ in $A$.

\begin{proof}[Proof of Theorem \ref{main}]
Fix a condensation $\varphi:C_p(X)\to (E,w)$. If instead of $C_p(X)$ we considered the space $C_p^\ast(X)$ of bounded continuous functions, the proof would remain exactly the same (simply by replacing $C_p(X)$ with $C_p^\ast(X)$ everywhere below).

Striving for a contradiction, suppose that $\varphi^{-1}:(E,w)\to C_p(X)$ is continuous at some point $e_0\in E$. Using appropriate translations, we can assume without loss of generality that $e_0=0$ and $\varphi(\underline{0})=0$. First, observe that $E$ is infinite-dimensional. Indeed, if $\dim E=k<\infty$, then $(E,w)$ is linearly homeomorphic to $\mathbb{R}^k$. Since $X$ is infinite, the space $C_p(X)$ contains a topological copy of the cube $[0,1]^{k+1}$, so it cannot be condensed onto $(E,w)$. Thus $E$ is infinite-dimensional. Let $K\subseteq E^\ast$ be a copy of $[0,1]$ provided by Lemma \ref{Klee-Mazur}.

For $k\in \omega$ and $m\geq 1$, define
$$Z_{k,m}=\left\{(F,y^\ast)\in [X]^k\times K: \varphi\big(O(F,m)\big)\subseteq W(y^\ast)\right\}.$$
Since $\varphi$ is surjective, we have $Z_{0,m}=\emptyset$ for all $m\geq 1$.

Let $\pi_2:[X]^k\times K\to K$ be the projection and set
$$C_{k,m}=\pi_2(Z_{k,m}).$$
Note that by the continuity of $\varphi$, we have
\begin{equation}\label{equation1}
 K=\bigcup\{C_{k,m}:k\in \omega, m\geq 1\},
\end{equation}
an observe that 
\begin{equation}\label{equation1a}
 C_{k,m} \subseteq C_{l,m}\quad \mbox{for } k,l \in \omega, k<l, m\geq 1.
\end{equation}

\begin{claimnumber}
For every $k\in \omega$, $m\geq 1$ and $F\in [X]^k$, the set $\{y^\ast\in K:(F,y^\ast)\in Z_{k,m}\}$ is finite.
\end{claimnumber}
\begin{proof}
Since $\varphi^{-1}$ is continuous at 0, there exist $y_1^\ast,\ldots , y_n^\ast\in E^\ast$ and $\varepsilon>0$ such that the following neighborhood
$$W'=\{e\in E:\forall i\leq n\;\; |y_i^\ast(e)|<\varepsilon\}$$
of 0 in $(E,w)$ satisfies
\begin{equation}\label{equation2}
\varphi^{-1}(W')\subseteq O(F,m).
\end{equation}
Take any $y^\ast\in K$ such that $(F,y^\ast)\in Z_{k,m}$, i.e.,
$\varphi(O(F,m))\subseteq W(y^\ast)$
(if no such $y^\ast$ exists, we are done, as the set considered in Claim 1 is empty).

We claim that $y\in \spa(y_1^\ast,\ldots , y_n^\ast)$. Indeed, otherwise
$$\bigcap_{i\leq n}\ker(y_i^\ast)\nsubseteq \ker(y^\ast).$$
where $\ker(z^\ast)$ denotes the kernel of the functional $z^\ast$ (cf. \cite[Lemma 3.9]{yellowbook})
This means that there is $y\in \bigcap_{i\leq n}\ker(y_i^\ast)$ with $y^\ast(y)\neq 0$. Scaling $y$ if necessary we can assume that $y^\ast(y)=2$. Consequently, $y\in W'$. Hence, by \eqref{equation2}, $\varphi^{-1}(y)\in O(F,m)$, whence $\varphi(\varphi^{-1}(y))=y\in W(y^\ast)$. This is a contradiction with $y^\ast(y)=2$.

Now, to conclude the proof of Claim 1 it suffices to note that the intersection $K\cap\spa(y_1^\ast,\ldots , y_n^\ast)$ is finite, because $K$ is linearly independent in $E^\ast$.
\end{proof}
By \eqref{equation1} and the Baire category theorem, there exists $k\in \omega$ and $m\geq 1$ such that
\begin{equation}\label{equation3}
\Int_K(\cl_K(C_{k,m}))\neq \emptyset.
\end{equation}
Let $k_0$ be the minimal $k$ such that \eqref{equation3} holds for some $m\geq 1$. Pick $m_0\geq 1$ such that
$$U_0=\Int_K(\cl_K(C_{k_0,m_0}))\neq\emptyset.$$
By minimality of $k_0$, we have $\Int_K(\cl_K(C_{k_{0}-1,m}))=\emptyset$ for all $m\geq 1$. Thus, we can find a nonempty open set $V\subseteq U_0$ satisfying
\begin{equation}\label{equation4}
V\cap C_{k_0-1,m_0}=\emptyset.
\end{equation}
Since $C_{k_0,m_0}$ is dense in $U_0$, which is open in a topological copy $K$ of $[0,1]$, we can find a one-to-one sequence $y^\ast_0, y^\ast_1,\ldots \in V\cap C_{k_0,m_0}$ norm-convergent to $y^\ast\in  (V\cap C_{k_0,m_0})\setminus\{y_0^\ast,y_1^\ast,\ldots\}$. For each $n\in \omega$, find $F_n\in [X]^{k_0}$ witnessing $y_n^\ast\in C_{k_0,m_0}$, that is,
\begin{equation}\label{F_n}
(F_n,y_n^\ast)\in Z_{k_0,m_0}.
\end{equation}
Since $y_0^\ast,y_1^\ast,\ldots$ are pairwise distinct, Claim 1 implies that the mapping\break $T:\{y_0^\ast,y_1^\ast,\ldots\} \to [X]^{k_0}$ defined by $T(y_n^\ast)=F_n$ has finite fibers. Hence, by passing to a subsequence, if necessary, we can without loss of generality assume that all sets $F_n$ are pairwise distinct.

Using a suitable version of the $\Delta$-System Lemma (see \cite[A.1.4]{vM}) and passing to a subsequence once more, we can further assume that there is a set $F\subseteq X$ with $|F|<k_0$ such that $F_n=F\cup H_n$, where $\{H_n:n\in \omega\}$ are nonempty and pairwise disjoint. Since $y^\ast\in V$ and $|F|<k_0$, it follows from \eqref{equation1a} and \eqref{equation4} that $y^\ast\notin C_{|F|,m}\subseteq C_{k_{0}-1,m}$. Hence, we can find $f\in O(F,m_0)$ such that $|y^\ast(\varphi(f))|>1$.

Applying Lemma \ref{lemma2} and passing to a subsequence, we obtain a sequence of continuous functions $f_n:X\to [0,1]$ such that $f_n(H_n)\subseteq \{1\}$ and the sequence $(f_n)_{n\in \omega}$ converges pointwise to $\underline{0}$.
Put
$$g_n=(1-f_n)\cdot f.$$
Clearly $g_n\rightarrow f$ pointwise so by the continuity of $\varphi$, we have $\varphi(g_n)\rightarrow \varphi(f)$ in $(E,w)$. In particular,
\begin{equation}\label{equation5}
 y^\ast(\varphi(g_n))\rightarrow y^\ast(\varphi(f)).
\end{equation}
Note also that $g_n(H_n)\subseteq \{0\}$ and $|g_n(x)|\leq |f(x)|<\tfrac{1}{m_0}$ for $x\in F$. Hence, $g_n\in O(F_n,m_0)$. Since $(F_n,y_n^\ast)\in Z_{k_0,m_0}$ (see \eqref{F_n}), we get
\begin{equation}\label{equation6}
 |y_n^\ast(\varphi(g_n))|\leq 1.
\end{equation}
Being a convergent sequence, the set $\{\varphi(g_n):n\in \omega\}$ is relatively compact in $(E,w)$ whence, it is norm-bounded. Since additionally $\|y^\ast-y_n^\ast\|\rightarrow 0$, we get
$$\lim_{n\to\infty}|y^\ast(\varphi(g_n))-y_n^\ast(\varphi(g_n))|=0.$$
This gives a desired contradiction since on the one hand $|y^\ast(\varphi(g_n))|\rightarrow |y^\ast(\varphi(f))|>1$, by \eqref{equation5}, while on the other hand $|y_n^\ast(\varphi(g_n))|\leq 1$ according to \eqref{equation6}.
\end{proof}

\begin{proof}[Proof of Proposition \ref{condensation may exist}]
We will construct condensations $\varphi:C_p(X)\to (E,w)$ and $\psi:C^*_p(X)\to (E,w)$ via compositions of seven condensations $\varphi=\varphi_7\circ \dotsb \circ \varphi_1$ and $\psi=\psi_7\circ \dotsb \circ \psi_1$. Let
$$\sigma=\{(x_n)\in\mathbb{R}^\omega: |\{n\in \omega:x_n\neq 0\}|<\aleph_0\}$$
be the subspace of the countable product of real lines consisting of all finitely supported sequences. Fix a countable dense subset $D$ of a $\sigma$-compact non-discrete metrizable space $X$. Let
$$C_D(X)=\{f\upharpoonright D:f\in C_p(X)\}\quad \text{and}\quad C_D^\ast(X)=\{f\upharpoonright D:f\in C_p^\ast(X)\}.$$
treated as subspaces of $\mathbb{R}^D$. Let $\varphi_1:C_p(X)\to C_D(X)$, respectively, $\psi_1:C_p^\ast(X)\to C_D(X)$ be condensations given by $f\mapsto f\upharpoonright D$.

By \cite[Theorem 6.2]{DM} both $C_D(X)$ and $C_D^\ast(X)$ are homeomorphic to $\sigma^\omega$. Let $\varphi_2:C_D(X)\to \sigma^\omega$ ($\psi_2:C_D^\ast(X)\to \sigma^\omega$) be a homeomorphism. It is readily seen that $\sigma^\omega$ is homeomorphic to $\sigma^\omega\times \mathbb{R}^\omega$. Fix a homeomorphism $\varphi_3:\sigma^\omega\to\sigma^\omega\times \mathbb{R}^\omega$. By Pytkeev's theorem (see \cite[Theorem 1]{Py}) there exists a condensation $\zeta:\sigma^\omega\to [0,1]^\omega$. Define $\varphi_4:\sigma^\omega\times\mathbb{R}^\omega\to [0,1]^\omega\times\mathbb{R}^\omega$ as $\varphi_4=\zeta\times id_{\mathbb{R}^\omega}$.

Using Toruńczyk's Factor Theorem (see \cite[Theorem 4.2.25]{vM} and \cite[Theorem 3.1]{To}), we can find a homeomorphism $\varphi_5:[0,1]^\omega\times \mathbb{R}^\omega\to \mathbb{R}^\omega$. According to the Anderson-Kadec Theorem (see \cite[Ch. VI, Theorem 5.2]{BP}), for every separable infinite dimensional Banach space $E$, there is a homeomorphism $\varphi_6:\mathbb{R}^\omega\to (E,\|\cdot\|)$. Finally, let $\varphi_7:(E,\|\cdot\|)\to (E,w)$ be the identity map. For $i\in\{3,\ldots , 7\}$, define $\psi_i=\varphi_i$ and set
\begin{eqnarray*}
\varphi &=& \varphi_7\circ \dotsb \circ \varphi_1:C_p(X)\to (E,w) \quad \mbox{and}\\
\psi &=& \psi_7\circ \dotsb \circ \psi_1:C_p^\ast(X)\to (E,w). 
\end{eqnarray*}
\end{proof}

\section{The pointwise topology vs. the weak* topology}
Making obvious modifications in the proof of Theorem \ref{main}, we get the following:
\begin{thrm}\label{main_weak^*}
Let $X$ be an infinite topological space and let $E$ be a normed space. If $\varphi:C_p(X)\to (E^\ast,w^\ast)$ is a condensation, then its inverse $\varphi^{-1}$ is discontinuous at every point of $E^\ast$. An analogous assertion holds if $C_p(X)$ is replaced by $C^*_p(X)$.
\end{thrm}
In particular, for an infinite space $X$ and a normed space $E$, the spaces $C_p(X)$ ($C_p^\ast(X)$) and $(E,w^\ast)$ are not homeomorphic.
However, it is worth pointing out that in this case, known results imply that there are topological properties distinguishing the pointwise topology from the weak* topology. The following proposition is a direct consequence of the Banach-Alaoglu Theorem and \cite[S.186]{Tk}
\begin{prop}\label{B_for_weak}
If $X$ is an infinite space and $E$ is a normed space, then $(E^\ast,w^\ast)$ is $\sigma$-compact while $C_p(X)$ is not.
\end{prop}

For the space $C_p^\ast(X)$ the situation is slightly more delicate. The following result is known and easily follows from \cite[S.396]{Tk}.

\begin{prop}
For any space $X$ the following conditions are equivalent.
\begin{enumerate}[(i)]
 \item $C_p^\ast(X)$ is $\sigma$-compact
 \item $X$ is discrete
\end{enumerate}
\end{prop}

From the above proposition and the Banach-Alaoglu Theorem we obtain:

\begin{cor}
If $X$ is an infinite non-discrete space and $E$ is a normed space, then $C_p^\ast(X)$ and $(E^\ast,w^\ast)$ are not homeomorphic, as one of the spaces is $\sigma$-compact and the other is not.
\end{cor}

For a discrete space $X$, both $C_p^\ast(X)$ and $(E^\ast,w^\ast)$ are $\sigma$-compact. However, in this case, one can use another topological property (the so-called property $(B)$) to distinguish between these spaces. Recall that a space $X$ has the \textit{property (B)} if there is a countable family $\{A_n:n\in \omega\}$ of closed nowhere dense subsets of $X$ such that for any compact subset $K$ of $X$, there is $n\in \omega$ with $K\subseteq A_n$ (see \cite[Definition 5.1]{KM}). The following fact was noted in \cite[Proposition 5.2]{KM}:
\begin{prop}
If $E$ is an infinite-dimensional normed space, then $(E^\ast,w^\ast)$ has the property (B).
\end{prop}

In \cite{KKM}, Kucharski and the present authors characterized those spaces $X$ for which $C_p(X)$ has the property $(B)$. To state this result, we first need to recall a few definitions.

Given a family $\mathcal{A}$ of subsets of a space $X$, we say that $\mathcal{A}$ is \textit{point-finite} if for each $x\in X$ the collection $\{A\in \mathcal{A}:x\in A\}$ is finite. The family $\mathcal{A}$ is \textit{strongly point-finite} if there is a point-finite family $\{U_A:A\in \mathcal{A}\}$ of open subsets of $X$ satisfying $A\subseteq U_A$ for every $A\in \mathcal{A}$.
Following \cite{Sa}, we say that a space $X$ has \textit{the property $(\kappa)$} if every  sequence of pairwise disjoint finite subsets of $X$ has an infinite strongly point-finite subsequence. A space $X$ is \textit{$\kappa$-Fr\'echet-Urysohn} if for every open subset $U$ of $X$ and every $x\in \overline{U}$, there is a sequence $\{x_n:n\in \omega\}\subseteq U$ that converges to $x$. Combining the results established in \cite{Sa} and \cite{KKM}, we have the following:

\begin{thrm}\label{theoremKKM}
For any space $X$ the following conditions are equivalent:
\begin{enumerate}[(i)]
 \item The space $X$ satisfies the property $\kappa$.
 \item The space $C_p(X)$ is $\kappa$-Fr\'echet-Urysohn.
 \item The space $C_p(X)$ does not have the property $(B)$.
\end{enumerate}
\end{thrm}
It turns out that a similar result holds for $C_p^\ast(X)$-spaces:

\begin{prop}\label{propertyB}
For any space $X$ the following conditions are equivalent:
\begin{enumerate}[(i)]
 \item The space $X$ satisfies the property $\kappa$.
 \item The space $C_p^\ast(X)$ is $\kappa$-Fr\'echet-Urysohn.
 \item The space $C_p^\ast(X)$ does not have the property $(B)$.
\end{enumerate}
\end{prop}
We need the following lemma.
\begin{lemma}\label{lemma}
Let $X$ be a space. If $C_p(X)$ is $\kappa$-Fr\'echet-Urysohn, then $C_p^\ast(X)$ is $\kappa$-Fr\'echet-Urysohn.
\end{lemma}
\begin{proof}
Fix an open subset $U$ of $C_p^\ast(X)$ and let $f\in \overline{U}$. In particular, $f$ is bounded, so there is $M>0$ such that $|f(x)|\leq M$ for all $x\in X$. Find an open subset $V$ of $C_p(X)$ such that
\begin{equation}\label{eq9}
U=V\cap C_p^\ast(X).
\end{equation}

Clearly, $f\in \overline{V}$ (this time the closure is taken in $C_p(X)$). Since $C_p(X)$ is $\kappa$-Fr\'echet-Urysohn there are $f_n\in V$, for $n\in \omega$, such that the sequence $(f_n)_{n\in \omega}$ converges pointwise to $f$. For each $n\in \omega$, there is a finite set $F_n\subseteq X$ and $\varepsilon_n>0$ such that the basic open neighborhood
$$O(f_n,F_n,\varepsilon_n)=\left\{g\in C_p(X): \forall x\in F_n \;\; |g(x)-f_n(x)|<\varepsilon_n \right\}$$
of $f_n$ in $C_p(X)$ satisfies
\begin{equation}\label{eq10}
 O(f_n,F_n,\varepsilon_n)\subseteq V.
\end{equation}

For $n\in \omega$, fix a real number $a_n$ such that
$$a_n>M+1\quad \mbox{and}\quad a_n>\max\{|f_n(x)|:x\in F_n\}.$$
Let $\rho_n:\mathbb{R}\to [-a_n,a_n]$ be a retraction. Define $g_n:X\to \mathbb{R}$ as
$$g_n=\rho_n \circ f_n.$$
Obviously,
$g_n\in C_p^\ast(X)$. Let us verify that $g_n\in O(f_n,F_n,\varepsilon_n)$. Indeed, pick $x\in F_n$. Since $a_n>|f_n(x)|$ and $\rho_n:\mathbb{R}\to [-a_n,a_n]$ is a retraction, we have $g_n(x)=\rho_n(f_n(x))=f_n(x)$. So $g_n$ agrees with $f_n$ on $F_n$, whence $g_n\in O(f_n,F_n,\varepsilon_n)$. Since $g_n\in C_p^\ast(X)$, we infer from \eqref{eq10} and \eqref{eq9} that $g_n\in U$.

It remains to show that the sequence $(g_n)_{n\in \omega}$ converges pointwise to $f$. To this end, fix $x\in X$. By the choice of the sequence $(f_n)_{n\in \omega}$, we have $f_n(x)\to f(x)\in [-M,M]$. Find $N$ such that $|f_n(x)|<M+1$ for all $n>N$. Since $a_n>M+1$ and $\rho_n:\mathbb{R}\to [-a_n,a_n]$ is a retraction, we have
$g_n(x)=\rho_n(f_n(x))=f_n(x)$ for all $n>N$. Since $f_n(x)\to f(x)$, we conclude that $g_n(x)\to f(x)$. This finishes the proof.
\end{proof}

\begin{proof}[Proof of Proposition \ref{propertyB}]
Assume that $X$ satisfies the property $(\kappa)$. By Theorem \ref{theoremKKM}, $C_p(X)$ is $\kappa$-Fr\'echet Urysohn. So according to Lemma \ref{lemma}, $C_p^\ast(X)$ is $\kappa$-Fr\'echet Urysohn. This shows $(i)\Rightarrow (ii)$. Now, if $C_p^\ast(X)$ is $\kappa$-Fr\'echet Urysohn, then by \cite[Theorem 3.2]{KKM}, $C_p^\ast(X)$ does not have the property $(B)$. Thus, $(ii)\Rightarrow (iii)$. Finally, suppose that $C_p^\ast(X)$ does not have the property $(B)$. Since $C_p^\ast(X)$ is a dense subspace of $C_p(X)$, we infer from \cite[Proposition 3.5]{Tk1}, that $C_p(X)$ does not have the property $(B)$. Hence, by Theorem \ref{theoremKKM}, $X$ satisfies the property $(\kappa)$. This proves $(iii)\Rightarrow (i)$.
\end{proof}

\begin{cor}
If $X$ is an infinite discrete space and $E$ is an infinite-dimensional normed space, then $C_p^\ast(X)$ and $(E^\ast,w^\ast)$ are not homeomorphic, as $(E^\ast,w^\ast)$ has the property $(B)$ while $C_p^\ast(X)$ does not have the property $(B)$.
\end{cor}
\begin{proof}
According to Proposition \ref{B_for_weak}, the space $(E^\ast,w^\ast)$ has the property $(B)$.
It is readily seen that any discrete space satisfies the property $(\kappa)$. So by Proposition \ref{propertyB}, the space $C_p^\ast(X)$ does not have the property $(B)$.
\end{proof}

Finally, let us note that by making obvious modifications to the proof of Proposition \ref{condensation may exist}, we obtain the following result related to Theorem \ref{main_weak^*}:

\begin{prop}\label{condensation may exist_weak^*}
Let $X$ be a $\sigma$-compact non-discrete metrizable space and let $E$ be an  infinite-dimensional Banach space with separable dual $E^\ast$. Then there exists a condensation $\varphi:C_p(X)\to (E^\ast,w^\ast)$. Similarly, there exists a condensation $\psi:C_p^\ast(X)\to (E^\ast,w^\ast)$.
\end{prop}

\section*{AI statement}

We have not used any artificial intelligence tools at any stage of this research or in the preparation of this paper.

\bibliographystyle{siam}
\bibliography{bib.bib}

\end{document}